\documentclass[11pt,fleqn]{article}

\newcommand{\versiondate}{(version February 23, 2017)} 
\usepackage{paralist}

\usepackage{a4}
\usepackage{amsthm}

\usepackage{ifthen}
\usepackage{amsmath}
\usepackage{amssymb}
\usepackage{graphicx}
\usepackage{color}
\usepackage{lscape}

\definecolor{darkgreen}{rgb}{0,.5,0}

\usepackage[colorlinks=true,linkcolor=blue,citecolor=darkgreen]{hyperref}

\swapnumbers
\theoremstyle{plain}
\newtheorem{theorem}{Theorem}[section]
\newtheorem{lemma}[theorem]{Lemma}
\newtheorem{proposition}[theorem]{Proposition}
\newtheorem{corollary}[theorem]{Corollary}

\theoremstyle{definition}
\newtheorem{definition}[theorem]{Definition}
\newtheorem{definitions}[theorem]{Definitions}

\newtheorem{remark}[theorem]{Remark}
\newtheorem{remarks}[theorem]{Remarks}
\newtheorem{subpart}[theorem]{} 

\newtheorem{application}[theorem]{Application}
\newtheorem{notation}[theorem]{Notation}
\newtheorem{monounaryalgebras}[theorem]{Monounary algebras}

\numberwithin{equation}{theorem} 

\newcommand{\New}[1]{\emph{#1}}

\DeclareMathOperator{\Sym}{Sym}
\DeclareMathOperator{\End}{End}

\DeclareMathOperator{\Con}{Con}

\DeclareMathOperator{\Pol}{Pol}
\DeclareMathOperator{\Inv}{Inv}

\DeclareMathOperator{\Eq}{Eq} 
\DeclareMathOperator{\Tol}{Tol}
\DeclareMathOperator{\Quord}{Quord}

\DeclareMathOperator{\preserves}{\triangleright}
\DeclareMathOperator{\notpreserves}{\not\triangleright}

\DeclareMathOperator{\At}{At}

\newcommand{\id}{{\sf id}}

\newcommand{\cE}{\mathcal{E}}
\newcommand{\cL}{\mathcal{L}}

\newcommand{\N}{\mathbb{N}}

\newcommand{\rmI}{{\rm I}}
\newcommand{\rmII}{{\rm II}}
\newcommand{\rmIII}{{\rm III}}
\newcommand{\rmpI}{{\rm (I)}} 
\newcommand{\rmpII}{{\rm (II)}}

\newcommand{\graph}[1]{#1^{\bullet}}

\newcommand{\meet}{\ensuremath{\land}}
\newcommand{\join}{\ensuremath{\lor}}

\newcommand{\filter}[2][r]{\ensuremath [#2\rangle_{#1}}
\newcommand{\kernelclass}[2][\ker f]{\ensuremath [#2]_{#1}}
\newcommand{\cover}{\prec}

\newcommand{\tra}{\textsf{tra}}
\newcommand{\sym}{\textsf{sym}}

\newcommand{\LE}[1][\cE]{\mathbf{0}_{#1}}
\newcommand{\GE}[1][\cE]{\mathbf{1}_{#1}}
\newcommand{\LT}[1][\cL]{\Delta_{#1}}
\newcommand{\GT}[1][\cL]{\nabla_{#1}}

\newcommand{\Pmod}[1]{(\text{mod}\,#1)}      
\newcommand{\Kong}[3]{#1\equiv #2 \, \Pmod{#3}} 

\let\rho=\varrho
\let\epsilon=\varepsilon
\let\phi=\varphi
\let\kappa=\varkappa

\let \restrictionORIGINAL=\restriction
\renewcommand{\restriction}{\hspace*{-0.9ex}\restrictionORIGINAL}

\author{Danica Jakub{\'\i}kov\'a-Studenovsk\'a\footnote{supported by
    Slovak VEGA grant 1/0063/14}\\ Institute of
  Mathematics\\ P.J. \v{S}af\'arik University\\ Ko\v sice
  (Slovakia)
\and Reinhard P\"oschel%
 \\Institute of Algebra\\TU Dresden (Germany)
\and S\'andor Radeleczki\addtocounter{footnote}{5}
 \footnote{This
  research started as part of the
  TAMOP-4.2.1.B-10/2/KONV-2010-0001 project, supported by the European
  Union, co-financed by the European Social Fund 113/173/0-2.} 
\\Institute of Mathematics\\ University of
    Miskolc (Hungary)}
\date{\emph{Dedicated to E. Tam\'as Schmidt}\\\versiondate}

\title{The lattice of congruence lattices of algebras\\
  on a finite set}

\begin{document}

\maketitle

\begin{abstract}
The congruence lattices of all algebras defined on a fixed
finite set $A$ ordered by inclusion form a finite atomistic lattice
$\cE$. We describe the atoms and coatoms.
Each meet-irreducible element of $\cE$ being determined by a single
unary mapping on $A$, we characterize completely those which are determined by a
permutation or by an acyclic mapping on the set $A$. 
Using these characterisations we
deduce several properties of the lattice $\cE$; in particular, we prove that
$\cE$ is tolerance-simple whenever $|A|\geq 4$.

\end{abstract}



\section{Introduction}

In 1963 \textsc{G. Gr\"atzer} and \textsc{E.T. Schmidt} proved that
every algebraic lattice is isomorphic to the congruence lattice of
some algebra (\cite{GraS63}). Since the algebras constructed by them were infinite,
the result immediately raised the question: Does every finite lattice occur
as the congruence lattice of a \textsl{finite} algebra? The problem
remained open till today, and it is usually referred as the
\textit{finite lattice representation problem}. It is an abstract
representation problem because it asks for a solution up to
isomorphism. The concrete version is the more involved question: Given a
sublattice $E$ of the the partition lattice $\Eq(A)$ of all
equivalence relations on a set $A$, does there exist an algebra on the same base
set $A$, such that $E$ \textsl{equals} the congruence lattice of this
algebra (in \cite{Wer76} such lattices $E$ are characterized by
closure properties).

The subject of the present paper is related to the finite
representation problem in its concrete version. For a fixed
finite set $A$ we consider all possible congruence lattices of
algebras with base set $A$. These congruence lattices (ordered by
inclusion) form itself a lattice $\cE$ and we are going to investigate
this lattice. An important tool is our
knowledge about the lattice $\cL$ of all quasiorder lattices of
algebras defined on the set $A$ described in \cite{JakPR2016} (using
some techniques developed previously in the papers \cite{JakPR11} and
\cite{JakPR13}). These two lattices are strongly
interrelated: there is a residual mapping from $\cL$ to
$\cE$. Therefore, in Section~\ref{sec:2}, we investigate on abstract
level, how lattice properties (which are relevant for us) behave under residual
mappings (for instance, the coatoms of $\cE$ directly can be obtained
from the coatoms of $\cL$, see~\ref{B1}\eqref{B1iv}).

Based on preliminary results from Section~\ref{sec:1} and the results
of \cite{JakPR2016} and Section~\ref{sec:2}, 
we describe the atoms (\join-irreducible
elements), coatoms (Section~\ref{sec:3}) and
further \meet-irreducible elements (Sections~\ref{sec:4}
and~\ref{sec:4a}) of the lattice $\cE$. 
Finally, in Section~\ref{sec:5}, we investigate several lattice
theoretic properties of $\cE$, e.g., it is tolerance simple, but has
no properties related with modularity.

\section{Preliminaries}\label{sec:1}

Throughout the paper we fix a base set $A$ (if not stated otherwise, 
$A$ is assumed to be finite). Further, let $\N:=\{0,1,2,\dots\}$ and
$\N_{+}:=\N\setminus\{0\}$. For a mapping $f\colon A\to A$, we write $fa$ for
the image of an element $a\in A$, and $f^{n}$ ($n\in\N$) denotes the $n$-fold
composition of $f$ (by convention, $f^{0}$ is the identity mapping
$\id_{A}$).

\begin{definitions}\label{A00}
Let $\Eq(A)$ and $\Quord(A)$ denote the set of all \emph{equivalence
relations} (reflexive, symmetric and transitive) and
\emph{quasiorders} (reflexive and transitive relations), respectively,
on a set $A$. 
The least and the greatest quasiorders (which are in fact
equivalences) are $\Delta:=\{(x,x)\mid x\in
A\}$ and $\nabla:=A\times A$.
 A unary mapping $f\colon A\to A$ \New{preserves} a quasiorder
 $q\in\Quord(A)$ (in
 particular, an equivalence $q=\kappa\in\Eq(A)$), 
notation $f\preserves q$, if
      \[ \forall x,y\in A: (x,y)\in q \implies (fx,fy)\in q.\]
This fact is also expressed by the following notions and notation: $f$
is an \New{endomorphism} of 
  $q$ ($f\in\End q$), $q$ is \New{invariant for} or \New{compatible
    with} $f$, or $q$ is a 
    \emph{quasiorder of $(A,f)$}
  ($q\in\Quord(A,f)$), or $\kappa$ is a \emph{congruence of
    $(A,f)$} ($\kappa\in\Con(A,f)$). 

The identity $\id_{A}\colon A\to A: x\mapsto x$ as well as all constant mappings
$A\to A: x\mapsto a$ are called \New{trivial} because just they
preserve every quasiorder $q\in \Quord(A)$.  
  For a unary algebra $(A,F)$, $F\subseteq A^{A}$, let $\Con(A,F)$ and
  $\Quord(A,F)$ be its \New{congruence} and \New{quasiorder lattice}, respectively, i.e.\ the lattice of all
  equivalences or quasiorders that are compatible with 
  each $f\in F$. Moreover, let 
  \begin{align*}
    \cE:=\{\Con(A,F)\mid F\subseteq A^{A}\} \text{ and }
    \cL:=\{\Quord(A,F)\mid F\subseteq A^{A}\}
  \end{align*}
  denote the lattice of all such congruence lattices and quasiorder
  lattices, respectively, on $A$, ordered
  by inclusion. Instead of $\Quord(A,F)$ and $\Con(A,F)$ we sometimes
  write simply $\Quord F$ and $\Con F$. Since congruences of an algebra
  are characterized by the unary polynomial functions of the algebra,
  the lattice $\cE$ is in fact the lattice of all congruence lattices
  of \textsl{arbitrary} (not necessarily unary) algebras on the set
  $A$ (the same holds for quasiorders and $\cL$).
\end{definitions}

\begin{remarks}\label{A01} 
The relation $\preserves$ induces (via the operators $\Con$ and
$\End$) a Galois connection between unary
mappings and equivalence relations on $A$. The Galois closures are just the
elements (congruence lattices) $\Con(A,F)\in\cE$ and monoids of the
form $\End Q$ (for some set $Q\subseteq\Eq(A)$), in particular, we have
\[E\in \cE\iff E=\Con(A,\End E) \text{ (i.e., $E$ is Galois closed).}\]
The meet in $\cE$ is the intersection while the join of elements
$E_{i}\in\cE$ ($i\in I$) is given by $\bigvee_{i\in
    I}E_{i}=\Con\End\bigcup_{i\in I} E_{i}$.

Clearly, $F\subseteq F'$ implies $\Con(A,F')\subseteq
\Con(A,F)$. Thus \meet-irreducibles in $\cE$ must be of the form
$\Con(A,f)$ for a single function $f$ because $\Con(A,F)$ is
the intersection of all $\Con(A,f)$ with $f\in F$. Analogously,
\join-irreducible (in case of infinite $A$, completely
\join-irreducible) elements of $\cE$ must be 
  of the form $E_{\kappa}:=\Con\End\kappa$ for a single equivalence relation
  $\kappa\in\Eq(A)\setminus\{\Delta,\nabla\}$,
  because, for $E\in \cE$, 
 $\End E$ is the intersection of all $\End \kappa$ and thus
 $\Con\End E=E$ is
the join (in $\cE$) of all $\Con\End \kappa$ with $\kappa\in E$.
  
\end{remarks}
\begin{notation}\label{A1a}
For $\kappa\in\Eq(A)$ consider the corresponding partition $A/\kappa$
into equivalence classes. If $C_{1}=\{a_{1},a_{2},\dots\}$,
$C_{2}=\{b_{1},b_{2},\dots\}$,\dots, $C_{m}=\{c_{1},c_{2},\dots\}$ are the
equivalence classes of $\kappa$ with at least two elements, then we
use the notation
\begin{align*}
  \kappa&=[a_{1},a_{2},\dots]\,[b_{1},b_{2},\dots]\,\dots\,
  [c_{1},c_{2},\dots]\text{ or }\\
 \kappa&=[C_{1}]\,[C_{2}]\,\dots\,[C_{m}].
\end{align*}
All other elements which do not appear in the list form one-element
equivalence classes. 
  
\end{notation}

\begin{monounaryalgebras}\label{A0}
Here we introduce some special notions for monounary algebras; for a
more general view to monounary algebras we refer to~\cite{JakP09}. 

Let $(A,f)$ be a finite
monounary algebra.
Let $Z_f(x):=\{f^ix\mid i\in\N\}$ be the
subalgebra of $(A,f)$ generated by an element $x\in A$. Obviously, we have $a\in Z_f(x)\iff
Z_f(a)\subseteq Z_f(x)$. 
We write $B\leq(A,f)$ if $B$ is (the carrier set of) a subalgebra of $(A,f)$.

Considering the graph $\graph f:=\{(a,b)\in A^{2}\mid b=fa\}$ of $f$,
one can use a graph theoretic terminology. 
For $a\in A$, let $K_f(a)$ denote the \New{connected component} of
$\graph f$ to
which $a$ belongs (note that two vertices $x,y\in A$ are connected w.r.t.~$f$,
iff there exist $i,j\in\N$ with
$f^{i}x=f^{j}y$). A component $K$ of $f$ is called \New{nontrivial} if
it contains at least two elements (thus a trivial component is just a
fixed point).

For a monounary algebra  $(A,f)$, the
  least quasiorder and congruence, resp.,
  containing a pair $(x,y)\in A^{2}$ is denoted by $\alpha_{f}(x,y)$
  and $\theta_{f}(x,y)$ 
 (\New{principal congruence}), resp., and we have
  \begin{align*}    
   \alpha_{f}(x,y)&=\Delta\cup\{(f^ix,f^iy)\mid i\in\N\}^{\tra},\\
    \theta_{f}(x,y)&=\Delta\cup\{(f^ix,f^iy)\mid i\in\N\}^{\sym\, \tra}.
  \end{align*}
Here $\Psi^{\sym}=\Psi\cup\Psi^{-1}$ denotes the symmetric closure and
$\Psi^{\tra}$  the transitive closure of a binary relation
$\Psi\subseteq A\times A$.
\end{monounaryalgebras}

We now collect some properties for functions $f,g\in A^A$ with
$\Con(A,f)\subseteq\Con(A,g)$.

\begin{lemma}\label{A2} Let $f,g\in A^A$ be nontrivial
  and $\Con(A,f)\subseteq\Con(A,g)$. Then we have

  \begin{enumerate}[\rm (i)]
  \item\label{A2i}
    $\forall x,y\in A: (x,y)\in \kappa\in\Con(A,f)\implies (gx,gy)\in \kappa$,\\
    in particular we have $(gx,gy)\in\theta_f(x,y)$ and 
    $\theta_{g}(x,y)\subseteq \theta_{f}(x,y)$.
  
  \item\label{A2ii} Let $B$ be a subalgebra of $(A,f)$. Then either $B$ is also
    a subalgebra of $(A,g)$ or $g$ is constant on $B$, where the
    constant does not belong to $B$. 
    \end{enumerate}
\end{lemma}

\begin{proof}
  \eqref{A2i} is clear since $f\preserves \kappa$ implies $g\preserves
  \kappa$ what follows from the
  assumption $\Con(A,f)\subseteq\Con(A,g)$.

\eqref{A2ii}: For a subalgebra $B$, $\epsilon_{B}:=\Delta\cup B\times B$
belongs to $\Con(A,f)$. Let $x\in B$. If $g$ is not constant on $B$,
then there exists  
$y\in B$ such that $gx\neq gy$. Because
$(x,y)\in \epsilon_{B}$, by (i) we
have $(gx,gy)\in \epsilon_{B}\setminus \Delta$, in particular $gx\in
B$. Thus $B$ is closed under $g$. If $g$ is constant on $B$ and $B$ is
not a subalgebra of $(A,g)$, then the constant cannot be an element of $B$.
 \end{proof}

 \begin{remark}\label{A2a}
   The property in \ref{A2}\eqref{A2i} completely characterizes the containment
   of the congruence lattices. We have for $f,g\in A^A$:
   \begin{align*}
     \Con(A,f)\subseteq\Con(A,g) \iff \forall x,y\in A:
     (gx,gy)\in\theta_{f}(x,y). 
   \end{align*}
 \end{remark}

In preparation of the next proposition we need the following lemma.

 \begin{lemma}\label{A3a}
   Let $f$ be  a permutation of 
   prime power order  $p^{m}$ with at least two cycles of length
   $p^{m}$. Then $\End\Con(A,f)=\End\Quord(A,f)$.
 \end{lemma}

 \begin{proof} The inclusion ``$\supseteq$'' is always true. To show
   ``$\subseteq$'', let $h\notin \End\Quord(A,f)$. Thus there exist
   $\rho\in\Quord(A,f)$ with $h\notpreserves\rho$ and therefore some
   principal quasiorder $\alpha_{f}(x,y)$ which is not preserved by
   $h$ for some $(x,y)\in \rho$. We must show $h\notin\End\Con(A,f)$. 
   Assume on the contrary that $h\in\End\Con(A,f)$ or, equivalently,
   $\Con(A,f)\subseteq\Con(A,h)$. 
W.l.o.g. we can
  assume $(hx,hy)\notin \alpha_{f}(x,y)$ (because there must exist
  $(u,v)\in \alpha_{f}(x,y)$ with $(hu,hv)\notin
  \alpha_{f}(u,v)\subseteq \alpha_{f}(x,y)$, one can use $(u,v)$
  instead of $(x,y)$).

If $x,y$ belong to the same cycle
 of the permutation $f$, then $\alpha_{f}(x,y)=\theta_{f}(x,y)$
 (cf.~\cite[Lemma~3.1]{Jak09}) and we
 have $h\notpreserves \theta_{f}(x,y)$, a contradiction.
 Thus we may assume $x\in C_{1}$, $y\in
 C_{2}$ where $C_{1}, C_{2}$ are different cycles of $f$ of length
 $p^{k_{1}}$ and $p^{k_{2}}$, resp. Moreover, w.l.o.g. assume
 $k_{1}\geq k_{2}$. Then we have
 $\alpha:=\alpha_{f}(x,y)=\Delta\cup\{(f^{i}x,f^{i}y)\mid
 0\leq i\leq p^{k_{1}}-1\}$ and 
$\theta:=\theta_{f}(x,y)=\alpha\cup\alpha^{-1}\cup\beta$ where 
   $\beta:=\{(f^{i}x,f^{j}x)\mid
   i,j\in\{0,1,\dots, p^{k_{1}}-1\},\,\Kong{j}{i}{p^{k_{2}}}\}$; note that
     $\beta\subseteq  C_{1}\times C_{1}$. 

We distinguish the following cases (recall $h\preserves\theta$ and
hence $(hx,hy)\in\theta\setminus\alpha$):

Case 1: $(hx,hy)\in \alpha^{-1}$, i.e., $(hx,hy)=(f^{i}y,f^{i}x)\in
C_{2}\times C_{1}$ for some $i$. By Lemma~\ref{A2}\eqref{A2ii}, $h$ is
constant $f^{i}x$ on $A\setminus C_{1}\leq (A,f)$ and constant $f^{i}y$ on
$A\setminus C_{2}\leq(A,f)$. If there exists $c\in
A\setminus(C_{1}\cup C_{2})$, then 
$f^{i}x=hc=f^{i}y$, a contradiction. Thus $A=C_{1}\cup C_{2}$,
i.e., $C_{1}$ and $C_{2}$ must be two cycles of length $p^{m}$. But then
$h$ does not preserve $\theta_{f}(x,fy)$ since
$(hx,hy)=(f^{i}y,f^{i}x)\notin\theta_{f}(x,fy)=[x,fy][fx,f^{2}y]\dots[f^{i}x,f^{i+1}y]\dots$, a contradiction.

Case 2: $(hx,hy)\in\beta$, hence
$(hx,hy)\in C_{1}\times C_{1}$. In particular we
 have $|C_{1}|>1$, i.e., $k_{1}\geq 1$. Further, 
by~\ref{A2}\eqref{A2ii}, $h$ is constant $hy\in C_{1}$ on $A\setminus
C_{1}\leq (A,f)$. 

If $k_{1}=k_{2}=m$, then
$(hx,hy)\in (C_{1}\times C_{1})\cap \theta_{f}(x,y)\subseteq\Delta$, a contradiction.

If $k_{2}< m$, then there exists a cycle $C$ of length
$p^{m}$ which is 
distinct from $C_{1}$. Let $x_{0}\in C$. Then each block of 
$\theta_{f}(x_{0},x)$ contains exactly one element of $C_{1}$, hence 
$(hx_{0},hx)=(hy,hx)\notin\theta_{f}(x_{0},x)$, a contradiction to
$h\preserves \theta_{f}(x_{0},x)\in\Con(A,f)$.

Thus the assumption $h\in\End\Con(A,f)$ must fail, i.e.,
$h\notin\End\Con(A,f)$.
 \end{proof}

 \begin{proposition}\label{A3}
   Let $f,g\in A^A$ be nontrivial such that
   $\Con(A,f)\subseteq\Con(A,g)$ and let $f$ be  a permutation of 
   prime power order  $p^{m}$ with at least two cycles of length
   $p^{m}$. Then  there
        exists $k\in\{1,\dots,p^{m}-1\}$ such that $g=f^k$.
 \end{proposition}

 \begin{proof}
   $\Con(A,f)\subseteq\Con(A,g)$ is equivalent to $g\in\End\Con(A,f)$,
   thus $g\in\End\Quord(A,f)$ by Lemma~\ref{A3a}, consequently
   $\Quord(A,f)\subseteq\Quord(A,g)$. From
   \cite[Proposition~2.5(b)]{JakPR2016} we conclude $\exists
   k\in\N_{+}: g=f^{k}$. Clearly $k$ can be chosen less than $p^{m}$
   since $f^{p^{m}}=\id_{A}$.
 \end{proof}

\section{Residual mappings and \meet-irreducibles}\label{sec:2}

We shall strongly use results about the lattice $\cL$ of quasiorder
lattices for the investigation of the lattice $\cE$ of congruence
lattices. However, we want to seperate those connections which are of pure
lattice theoretic nature (and which are -- from our point of view --
of its own interest). This is done in this section. Based on the
observation that
 $\Phi: \cL\to\cE: Q\mapsto Q\cap\Eq(A)$
is a residual mapping,
we consider this case in a general setting.

Let $L$ and $E$ be arbitrary lattices which -- for simplicity -- here are assumed
to be finite, the least and largest elements are denoted by $0_{L},
0_{E}$ and $1_{L},1_{E}$.
A mapping $\phi:L\to E$ is called \New{residual}  if it
is a \meet-homomorphism (and therefore also monoton with respect to the
lattice orders) and $\phi(1_{L})=1_{E}$ (cf.\ e.g.\
\cite{Jan94} or \cite{JanR2015}). 

The following proposition
shows that then the \meet-irreducible elements of $\cE$, in particular
coatoms, can be
constructed from the \meet-irreducible elements of $\cL$.

\begin{proposition}\label{B1} Let $\phi:L\to E$ be a surjective
  residual mapping. 
  \begin{enumerate}[\rm(i)]
\item \label{B1i} Let $m\in
E$ be a \meet-irreducible element in $E$. Then
$\phi^{-1}(m)\neq\emptyset$ and each $q\in L$ which is maximal in
$\phi^{-1}(m)$ is \meet-irreducible in $L$.

\item\label{B1ii} Let $m\in E$ and let each $q\in L$ with $\phi(q)=m$ be
  \meet-irreducible in $L$. Then $m$ is \meet-irreducible in $E$.

\item \label{B1iii} Assume 
  \begin{align*}\tag{\textdagger}
    \phi(x)=1_{E}\implies x=1_{L} \text{ for all } x\in L.
  \end{align*}

Then for each coatom $m\in E$ there exists a coatom in 
 $q\in L$ such that $\phi(q)=m$.

\item \label{B1iv} Assume condition $(\dagger)$ above and
  \begin{align*}\tag{\ddag}
    \phi(q)\leq\phi(q')\implies \phi(q)=\phi(q') \text{ for all
      coatoms } q,q'\in L.
  \end{align*}
Then $\phi(q)$ is a coatom in $E$ if $q$ is a
coatom  in $L$. Moreover, the set of all coatoms of $E$ is $\{\phi(q)\mid q \text{ coatom in } L\}$.

 \end{enumerate}
\end{proposition}

\begin{proof}
  \eqref{B1i}: Let $m\in E$ be \meet-irreducible. Then
  $\phi^{-1}(m)=\{q'\in L\mid \phi(q')=m\}$ is nonempty because $\phi$
  is surjective. Let $q$ be maximal in
  $\phi^{-1}(m)$. Then $q$ is the
  meet of \meet-irreducible elements, say $q=q_{1}\meet \dots\meet
  q_{s}$ with \meet-irreducible $q_{i}\in L$ ($i\in\{1,\dots,s\}$).
  It follows $m=\phi(q)=\phi(q_{1})\meet\dots\meet \phi(q_{s})$.
  Because $m$ is \meet-irreducible there exists $i\in\{1,\dots,s\}$
  such that $m=\phi(q_{i})$. Since $q\leq q_{i}$ and $q$ was chosen
  maximal with respect to $\phi(q)=m$, we have $q=q_{i}$, i.e., it is
  \meet-irreducible.

\eqref{B1ii}:  Let $m=m_{1}\land m_{2}$ for some $m_{1}, m_{2}\in E$. Since $\phi$ is
surjective, there exist $q_{i}\in L$ with $\phi(q_{i})=m_{i}$,
$i\in\{1,2\}$. Let $q:=q_{1}\land q_{2}$. Then
$\phi(q)=\phi(q_{1}\land q_{2})=\phi(q_{1})\land
\phi(q_{2})=m_{1}\land m_{2}=m$ and $q$ must be \meet-irreducible by
assumption. Consequently, there is $i\in\{1,2\}$ with $q=q_{i}$, thus
$m=\phi(q)=\phi(q_{i})=m_{i}$, i.e., $m$ is \meet-irreducible.

\eqref{B1iii}: Since $L$ is finite, by \eqref{B1i} there exists a
maximal \meet-irreducible $q\in L$ with 
$\phi(q)=m$. If $q$ were not a coatom then there would exist a 
$q'\in L$ with $q<q'<1_{L}$. By the maximality property of $q$, we
get $\phi(q')>m$, thus $\phi(q')=1_{E}$ (since $m$ is coatom) and by
the assumption from \eqref{B1iii} we would get $q'=1_{L}$, a contradiction.

\eqref{B1iv}: Let $q\in L$ be a coatom. Then $\phi(q)\neq 1_{E}$ because of
$(\dagger)$. Thus there exists some coatom $m$ in $E$ with
$\phi(q)\leq m$. By \eqref{B1iii} there exists a coatom $q'$ in $L$ such that
$\phi(q')=m$. Then $\phi(q)\leq\phi(q')$ and with (\ddag) we get that
$\phi(q)=\phi(q')=m$ is a coatom in $E$. This together with \eqref{B1iii} shows
that $\{\phi(q)\mid q \text{ coatom in } L\}$ is the set of all
coatoms of $E$.
\end{proof}

\begin{remark}\label{B1a} Concerning~\ref{B1}\eqref{B1i}, since $L$ is
  finite, for any
  $q'\in\phi^{-1}(m)$ there exists a maximal (and therefore
  \meet-irreducible) $q$ with $q'\leq q\in\phi^{-1}(m)$.
  
\end{remark}

\begin{application}\label{B2}
There are many applications of residual mappings in various contexts,
in particular in connection with the unique corresponding residuated
mapping (establishing a ``covariant Galois connection''). However, for
this paper the only example which we need is the above mentioned
residual mapping
\begin{align*} \label{Phi}\tag{*}
  \Phi:\cL\to\cE: Q\mapsto Q\cap\Eq(A)\,,
\end{align*}
(recall $\cL:=\{\Quord(A,F)\mid F\subseteq A^{A}\}$ and
$\cE:=\{\Con(A,F)\mid F\subseteq A^{A}\}$ from~\ref{A00}). Clearly,
$\Phi(\Quord(A,F))=\Con(A,F)\in\cE$. The next Lemma shows that the
assumptions (\textdagger) and (\ddag) in
Proposition~\ref{B1}\eqref{B1iii},\eqref{B1iv} are
satisfied for this 
example. Notice that $\Eq(A)$ and $\Quord(A)$ are the greatest elements
of the lattices $\cE$ and $\cL$, respectively.
\end{application}

\begin{lemma}\label{B3}
  \begin{itemize}
 \item[\rm (i)] For $Q\in\cL$, $\Phi(Q)=\Eq(A)$ implies $Q=\Quord(A)$.
  \item[\rm (ii)] For coatoms $Q,Q'$ in $\cL$,
    $\Phi(Q)\subseteq\Phi(Q')$ implies $\Phi(Q)=\Phi(Q')$.
  \end{itemize}
\end{lemma}

\begin{proof}
  (i): It is well-known that trivial functions ($\id_{A}$ and the
  constants, say $C$) are the only mappings which
  preserves all equi\-va\-lence relations. Thus we have
 $Q=\Quord\End Q\supseteq
 \Quord\End\Phi(Q)=\Quord\End\Eq(A)=\Quord(\{\id_{A}\}\cup C)=\Quord(A)$. 

(ii): This will follow immediately from Proposition~\ref{P2} proved in
Section~\ref{sec:3}. 
\end{proof}

Note that the coatoms of $\cE$ and $\cL$ are of the form $\Con(A,f)$
and $\Quord(A,f)$ for some specific $f$ (they are of type (I)-(III) as
we shall see in Theorem~\ref{Thm1}). Thus from~\ref{B1}\eqref{B1iv} we
immediately get: 

\begin{corollary}\label{B4} $\{\Con(A,f)\mid \Quord(A,f) \text{ is a
    coatom in } \cL\}$ is the set of all coatoms of $\cE$.\qed
\end{corollary}

We close this section with two results which shall turn out to be
useful later. For an element $C$ of a lattice $\cE$ let
$\filter[\cE]{C}:=\{C'\in \cE\mid C\leq C'\}
 $ denote the principal filter generated by $C$.

\begin{lemma}\label{B5}
For $C\in\cE$ and
$Q:=\Quord\End C$ we have
$\Phi^{-1}(\filter[\cE]{C})=\filter[\cL]{Q}$.
\end{lemma}

\begin{proof}
Recall $C\in \cE\iff \Con\End C=C$ and $Q'\in\cL\iff \Quord\End
Q'=Q'$. Now,
  if $Q'\in\filter[\cL]{Q}$ then
   $\Phi(Q')\in \filter[\cE]{\Phi(Q)}=\filter[\cE]{C}$ since
  $\Phi$ is order preserving and $\Phi(Q)=\Eq(A)\cap\Quord\End
  C=\Con\End C = C$. Thus $Q'\in
   \Phi^{-1}(\filter[\cE]{C})$. 
Conversely, if $Q'\in \Phi^{-1}(\filter[\cE]{C})$, then
$C\subseteq \Phi(Q')\subseteq Q'$ and we get $Q=\Quord\End C\subseteq
\Quord\End Q'=Q'$, i.e., $Q'\in\filter[\cL]{Q}$.
\end{proof}

\begin{corollary}\label{B6}
For $F\subseteq A^{A}$ we have
  $\End\Quord(A,F)=\End\Con(A,F)$ if and only if
  $\Phi^{-1}(\filter[\cE]{\Con(A,F)})=\filter[\cL]{\Quord(A,F)}$. 
\end{corollary}

\begin{proof}
  Let $Q:=\Quord(A,F)$ and $C:=\Con(A,F)$. Then we have
  $C=\Phi(Q)\subseteq Q$.
 If $\End Q=\End C$ then $Q=\Quord\End Q=\Quord\End C$ and we get 
$\Phi^{-1}(\filter[\cE]{C})=\filter[\cL]{Q}$ from
Lemma~\ref{B5}. Conversely,
assume $\Phi^{-1}(\filter[\cE]{C})=\filter[\cL]{Q}$ and let
$Q':=\Quord\End C$. Then $\Phi(Q')=\Con\End C=C$, thus
$Q'\in\Phi^{-1}(\filter[\cE]{C})$. Consequently,
$Q'\in \filter[\cL]{Q}$, i.e., $Q\subseteq Q'=\Quord\End C$, and we
get $\End Q\supseteq \End\Quord\End C\supseteq\End C\supseteq\End Q$
(the latter inclusion follow from $C\subseteq Q$),
hence $\End Q=\End C$.
\end{proof}

\section{Atoms and coatoms of $\cE$} \label{sec:3}

In this section we are going to describe the atoms (\join-irreducibles) and coatoms of $\cE$.
The case $|A|=2$ is trivial. Then $\cE$ consists only of one lattice,
namely
$\Con(A,A^{A})=\Con(A,\id_A)=\Eq(A)=\{\Delta,\nabla\}$. Therefore, in
the following we assume always $|A|\geq 3$. 

The \join-irreducibles  
are easily described. In the following theorem, $A$ may be an
 arbitrary, not necessarily finite, set.

\begin{theorem}\label{Thm2}
The completely \join-irreducibles of $\cE$ are exactly the congruence
lattices of the form
\begin{align*}
\tag{*}
E_{\kappa}:=\Con(A,\End \kappa)
  =\{\Delta, \kappa, \nabla\}  
\end{align*}
where $\kappa\in\Eq(A)\setminus\{\Delta,\nabla\}$ is an arbitrary
equivalence relation. Moreover, each \join-irreducible is an atom in
$\cE$, i.e. the lattice $\cE$ is atomistic.
\end{theorem}

\begin{proof}  Completely \join-irreducibles must be
  of the form $E_{\kappa}$ as noted in~\ref{A01}.
 The characterization (*) follows immediately from \cite[Corollary
 2.5]{PoeR08}, where it is shown that $\Quord(A,\End \kappa)=\{\Delta,
 \kappa, \nabla\}$ (and therefore equals $\Con(A,\End \kappa)$) for
 any equivalence relation $\kappa$. Clearly, such lattices are atoms
 in $\cE$ 
 and therefore \join-irreducible,
 since $\{\Delta, \nabla\}$ is the only proper sublattice.
\end{proof}

\begin{remark}\label{A4}
For $L\in\cE$ let $\At(L):=\{E_{\kappa}\mid E_{\kappa}\subseteq L,\,
\kappa\in\Eq(A)\}=\{E_{\kappa}\mid \kappa\in L\setminus\{\Delta,\nabla\}\}$ be the set of atoms contained in $L$. It is natural
to ask which sets of atoms are of the form $\At(L)$ for some
$L\in\cE$. Equivalently, for given $E\subseteq \Eq(A)$, we may ask for
$\At(\Con(A,\End E))$. Formally we put
$[E]:=\{\Delta,\nabla\}\cup\At(\Con(A,\End E))$ because then $E\mapsto [E]$
is a closure operator which is well-known: $[E]$ coincides with the
Galois closures of the Galois connection $\End-\Inv$ as well as of
$\Pol-\Inv$ restricted to equivalence relations (because $\Con F =\Inv
F\cap \Eq(A)$), cf.~\cite{PoeK79} or \cite{Poe04a}, and can be
explicitly described by so-called graphical compositions as shown by
\textsc{H. Werner} in~\cite{Wer76}.

\end{remark}

Now we describe the coatoms.

\begin{theorem}\label{Thm1}
    The coatoms of $\cE$ are exactly the congruence lattices of the form
  $\Con(A,f)$ where $f\in A^A$ satisfies
  \begin{itemize}
\item[\rm(I)] $f$ is nontrivial and $f^2=f$, or
\item[\rm(II)] $f$ is nontrivial, $f^2$ is a constant, say $0$, and
  $|\kernelclass 0|\geq 3$, or 
\item[\rm(III)] $f^p=\id_A$ for some prime $p$ such that the
  permutation $f$ has at least two cycles of length $p$.
  \end{itemize}
\end{theorem}

Remark: It can happen that different functions $f$ from the theorem
give the same coatom. This explicitly will be clarified in
Proposition~\ref{P2}. The theorem describes three types of functions,
the graphs of which are shown in Figure~\ref{fig:1} (all labeled
elements are mandatory, all others are optional). Moreover, there are
functions which define coatoms but which are of none of the types
(I)-(III), e.g., on $A=\{0,1,2\}$ we get a coatom $\Con(A,f)=\Con(A,g)$ for
$f:1\mapsto 0\mapsto 0$, $2\mapsto 2$ and $g:1\mapsto 2\mapsto
0\mapsto 2$ where $f$ is of type (I), but $g$ is of no type.

\begin{figure}[htb]
  \centering
  \input{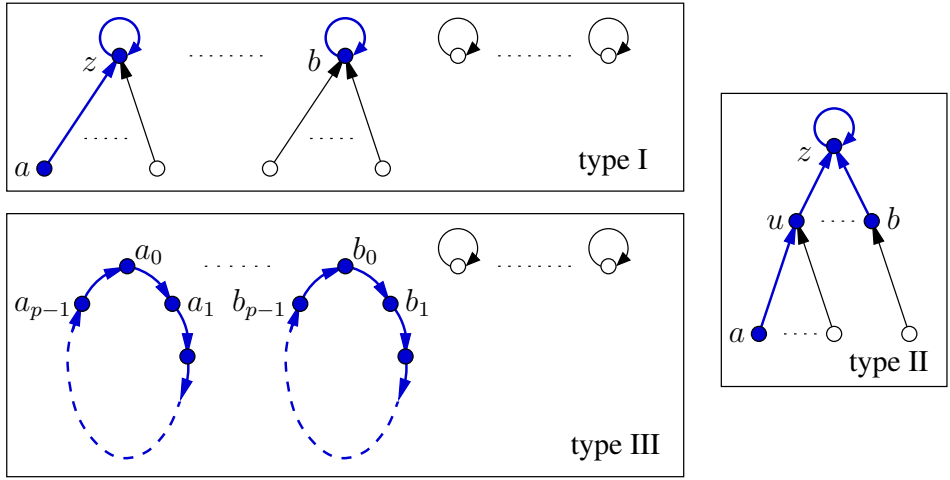}
  \caption{The graphs of functions of type I, II and III}
  \label{fig:1}
\end{figure}

\begin{proof} The coatoms in the lattice $\cL$ of quasiorder lattices
 are known from \cite[Theorem~3.1]{JakPR2016} and can be
described exactly as $\Quord(A,f)$ for the nontrivial functions $f$ of the three
types (I)-(III). Thus the Theorem immediately follows
from Corollary~\ref{B4}.
\end{proof}

Note that the proof is based on Corollary~\ref{B4} which follows from
\ref{B1}\eqref{B1iv} and needs the condition (\ddag) in its concrete form in
Lemma~\ref{B3}(ii). This will be proved with Proposition~\ref{P2}
below. Moreover, in~\ref{P2} it will 
be clarified when two different functions of type (I)-(III) give the
same congruence lattice $\Con(A,f)$ (by \ref{B1}\eqref{B1iii} and \ref{B3}(i) we 
already know that each coatom in $\cE$ must be of the form $\Phi(Q)=\Con(A,f)$
for some function $f$ of type (I)-(III)).
However, before stating \ref{P2} we need some more notions, 
notations and a lemma. 

\begin{subpart}  \label{N1}
If $f\in A^{A}$ is of type (I) and has exactly one nontrivial
  component $K_{f}(z)$ with fixed point $z$, then let $\hat f$ be defined by
  \begin{align*}
   {\hat f}x:= \begin{cases}
                z&\text{if } fx=x\\
                x&\text{otherwise},
                \end{cases}
  \end{align*}
see Figure~\ref{fig:1A}. If $f\in A^{A}$ is of type (II) with
two-element image $f[A]=\{u,z\}$, where $z$ shall denote the fixed
point, then let $\hat f$ be defined by
  \begin{align*}
   {\hat f}x:= \begin{cases}
                z&\text{if } fx=u\\
                u&\text{otherwise}.
                \end{cases}
  \end{align*}

In all other cases we put $\hat f:=f$. From the Figure~\ref{fig:1A}
is clear that $\hat{\hat f}=f$, moreover $(A,f)$ and $(A,\hat f)$ have
the same principal congruences, thus $\Con(A,f)=\Con(A,\hat f)$.

\end{subpart}

\begin{figure}[htb]
  \centering
  \input{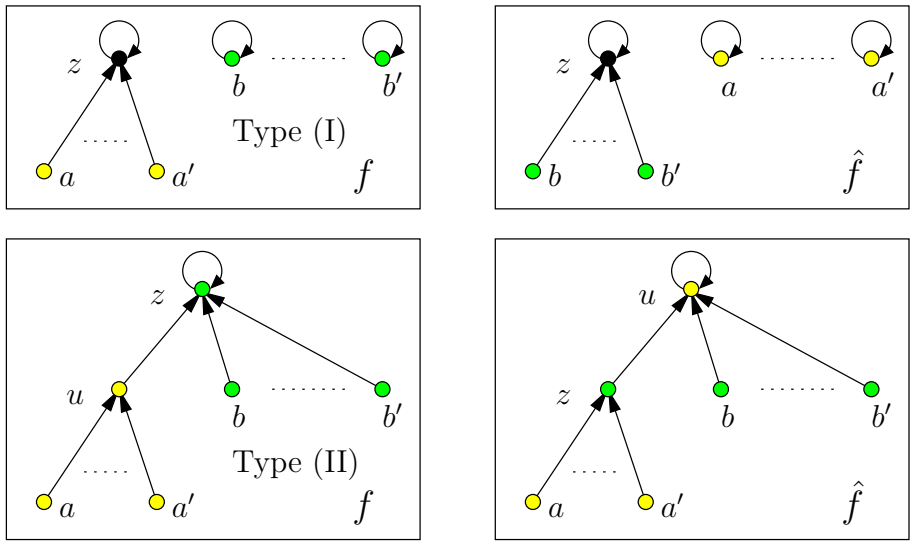}
  \caption{The functions $\hat f$}
  \label{fig:1A}
\end{figure}

\begin{definition} 
For a function $g\in A^{A}$ of type
  (I) or (II), respectively, a triple $(x,z,y)$ of three different
  elements is called
  \New{essential} for $g$ (or 
  \New{$g$-essential}) if $gx=z=gz$ and $gy=y$, or $gx=y$ and $gy=z=gz$,
  respectively, see Figure~\ref{fig:Es}.
\end{definition}

\begin{figure}[htb]
  \centering
  \input{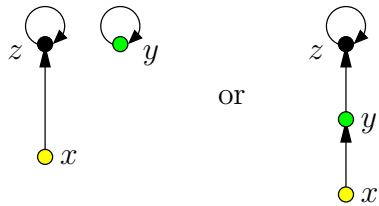}
  \caption{The graph of $g\protect\restriction_{\{x,y,z\}}$ for a $g$-essential
    triple $(x,z,y)$}
  \label{fig:Es}
\end{figure}

\begin{lemma}\label{L:Es} Let $\Con(A,f)\subseteq\Con(A,g)$ for
  functions $f,g\in A^{A}$, and let three different elements $x,y,z\in
  A$ satisfy $(x,z)\in \theta_{g}(x,y)$. Then we have:
  \begin{itemize}
  \item[\rm (a)] If $f$ is of type \rmpI, then $(x,z,y)$ or $(y,z,x)$ is
    $f$-essential. 
  \item[\rm (b)] If $f$ is of type \rmpII, then $(x,y,z)$ or $(x,z,y)$
    or $(y,x,z)$ or $(y,z,x)$ is $f$-essential (note that $z$ never
    appears in the first component of these triples).
  \end{itemize}  
\end{lemma}

\begin{proof}
  Because of Lemma~\ref{A2}(\ref{A2i}) we have
  $(x,z)\in\theta_{g}(x,y)\subseteq \theta_{f}(x,y)$. For functions of
  type (I) or (II) the only possibilities that all three elements $x,y,z$ belong
  to the same block of $\theta_{f}(x,y)$ are those mentioned in
  (a) and (b), cf.\ Figure~\ref{fig:Es} and Figure~\ref{fig:1}.
\end{proof}

\begin{remark} If $g\in A^{A}$ is of
  type \rmpI\ or \rmpII, then 
  the $g$-essential triples $(x,z,y):=(a,z,b)$ or $(x,z,y):=(a,z,u)$
  (notation as in Figure~\ref{fig:1}),
  respectively, satisfy the condition $(x,z)\in \theta_{g}(x,y)$ from
  Lemma~\ref{L:Es}. 
\end{remark}

\begin{proposition}\label{P2}
Let $f,g\in A^{A}$ be nontrivial operations of one of the types
{\rm (I)-(III)} such that $\Con(A,f)\subseteq \Con(A,g)$. Then $g\in\{f,\hat f\}$ if $f$ is of type
{\rm (I)} or {\rm (II)}, and $g=f^{i}$ for some $i\in\{1,\dots,p-1\}$
if $f$ is of type {\rm (III)}. In particular we always have
$\Con(A,f)=\Con(A,g)$.    
\end{proposition}

\begin{proof}
If $f$ is of type (III), then from Proposition~\ref{A3}
(where functions $f$ of prime power order $p^{m}$ are considered, here
one has to take $m=1$) we conclude that
$g=f^{k}$ for some $k\in \{1,\dots,p-1\}$.
  Corresponding to the types of
$f$ and $g$, there remain altogether 6 cases to consider, denoted by
$(X,Y)$ if $f$ is of type $(X)$ and $g$ of type $(Y)$, where
$X\in\{\rmI, \rmII\}$ and $Y\in\{\rmI, \rmII, \rmIII\}$. We start
with the cases where $f$ and 
$g$ have different types. For these cases we shall indicate elements $x,y\in A$
with $\theta_{g}(x,y)\not\subseteq\theta_{f}(x,y)$, a contradiction
to~\ref{A2}(\ref{A2i}); 
e.g. by using Lemma~\ref{L:Es} or by finding $(x,z)\in
\theta_{g}(x,y)\setminus\theta_{f}(x,y)$.  

\underline{Case (\rmI, \rmII):} Take the $g$-essential triple
$(x,z,y)=(a,z,u)$. By \ref{L:Es}(a) we have that $(a,z,u)$ or
$(u,z,a)$ is $f$-essential. Thus $(a,u)\notin\theta_{f}(a,z)=[a,z]$, in
contradiction to $(a,u)\in\theta_{g}(a,z)\subseteq\theta_{f}(a,z)$. 

\underline{Cases (\rmI, \rmIII) and (\rmII, \rmIII):} Since $g$ is of
type (III), there exist 
elements $a_{0},b_{0}$ with $\theta_{g}(a_{0},b_{0})=[a_{0},b_{0}][a_{1},b_{1}]\ldots[a_{p-1},b_{p-1}]$ (see
Figure~\ref{fig:1}). Because $\theta_{g}(a_{0},b_{0})\subseteq
\theta_{f}(a_{0},b_{0})$ we get that $f$ (as function of type (I) or (II)) must
satisfy $\{fa_{0},fb_{0}\}=\{a_{1},b_{1}\}$ (moreover, $p=2$) and
$\theta_{f}(a_{1},b_{1})=[a_{1},b_{1}]$ (because $f^{2}=f$ for type
(I), and $f^{2}a_{0}=f^{2}b_{0}=z$ for $f$ of type (II)); in
particular $\theta_{f}(a_{1},b_{1})$ cannot contain
$\theta_{g}(a_{1},b_{1})=\theta_{g}(a_{0},b_{0})$, a contradiction.

\underline{Case (\rmII, \rmI):} Take the $g$-essential triple
$(x,z,y):=(a,z,b)$. By \ref{L:Es}(b) there exists an $f$-essential
triple  $(a',z',u')$ with $\{a,b,z\}=\{a',u',z'\}$. Further, there
must exist $b'\notin\{a',u',z'\}$ with $fb'=z'$.
Note that $\theta_{f}(x,b')$ consists only of 2-element
blocks for each $x\in\{a',u',z'\}=\{a,b,z\}$. Therefore,
$\theta_{g}(a,b')=[a,b',z]$ if $gb'=b'$, or
$\theta_{g}(b,b')=[b,b',gb']$ if $gb'\neq b'$, cannot be contained
in $\theta_{f}(a,b')$ or $\theta_{f}(b,b')$, respectively, a contradiction.

Now we continue with the cases where $f$ and $g$ are of the same type.

\underline{Case (\rmI, \rmI):} Let $(a,z,b)$ be an $g$-essential triple.
According to Lemma~\ref{L:Es}(a), we have only the following two
possibilities: 
\begin{align*}
  \tag{*} fa&=z \text{ and } fb=b\\
  \tag{**} \text{ or } fb&=z \text{ and } fa=a.
\end{align*}
Clearly this must hold for \textsl{each} essential triple of $g$.
We show that (*) implies $g=f$ and (**) implies $f=\hat g$.
Let $c\in A\setminus\{a,b,z\}$ arbitrary.

Consider case (*): If
$gc=c$, take the $g$-essential triple $(a,z,c)$ for  
which only case (*) is possible (since $fa=z$), i.e. $fc=c=gc$. If
$gc\neq c$, take the $g$-essential
triple $(c,gc,b)$ (provided that $gc\neq b$) or $(c,gc,z)$ (if $gc=b$)
for which again case (*) must hold (since $fb=b$ and $fz=z$),
i.e. $fc=gc$, too. Altogether $f=g$.

Now consider case (**): If $gc=c$, take the $g$-essential triple $(a,z,c)$ for
which only case (**) is possible (since $fa=a$), i.e. $fc=z$.
If
$gc=z\neq c$, take the essential
triple $(c,z,b)$ for which again case (**) must hold (since $fb=z$),
i.e. $fc=c$. 
The case $gc\notin\{c,z\}$ cannot appear because then 
we would have $[a,c][z,gc]=\theta_{g}(a,c)\subseteq
\theta_{f}(a,c)=[a,c,fc]$, a contradiction. Thus altogether we have 
 $f=\hat g$.

\underline{Case (\rmII, \rmII):} Let $(a,z,u)$ be a $g$-essential triple.
Then the triples $(x,z',y)=(a,z,u)$ and $(x,z',y)=(a,u,z)$ satisfy the
assumptions of Lemma~\ref{L:Es}(b), consequently neither $z$ nor $u$
can be in the first component of the corresponding $f$-essential
triples. Thus there remain only the following two
possibilities: 
\begin{align*}
  \tag{*} fa&=u \text{ and } fu=z=fz\\
  \tag{**} \text{ or } fa&=z \text{ and } fz=u=fu.
\end{align*}
Clearly this must hold for \textsl{each} essential triple of $g$.
We show that (*) implies $g=f$ and (**) implies $f=\hat g$
(equivalentyl, $g=\hat f$).

Consider case (*): Since $fz=z$ is the unique fixed point of $f$, for
each $g$-essential triple $(a',z,u')$ we get case (*),i.e. $(a',z,u')$ is $f$-essential and thus $f$ and $g$ agree on all
essential triples. If $b$ does not belong to an $g$-essential triple,
then $gb=z$ and $\theta_{g}(a,b)=[a,b][u,z]$. From
$\theta_{g}(a,b)\subseteq\theta_{f}(a,b)=[a,b][u,fb]$ we conclude
$fb=z$; altogether $f=g$.

Consider now case (**):  Since $fu=u$ is the unique fixed point of
$f$ but $fz=u$ is not a fixed point, for
each $g$-essential triple $(a',z,u')$ we must have case (**),
i.e. $(a',u',z)$ is $f$-essential, in particular $u'=u$ and
$fa'=z$.  
If $b$ does not belong to an $g$-essential triple,
then $gb=z$ and $\theta_{g}(a,b)=[a,b][u,z]$. From
$\theta_{g}(a,b)\subseteq\theta_{f}(a,b)=[a,b][z,fb]$ we conclude
$fb=u$; altogether $f=\hat g$.
\end{proof}

\section{\meet-irreducible $\boldsymbol{\Con(A,f)}$ in $\boldsymbol{\cE}$ with
  permutation $\boldsymbol{f}$}\label{sec:4}

The coatoms are \meet-irreducible in $\cE$. Now we want to deal with
\meet-irreducible congruence lattices in general. 
 They all must be of the form
$\Con(A,f)$ for a single $f$ and we have:

\begin{proposition}\label{P3}
Let $\Con(A,f)$ be a \meet-irreducible element in $\cE$. Then there exists $g\in
A^{A}$ such that $\Quord(A,g)$ is \meet-irreducible in $\cL$ and
$\Con(A,f)=\Con(A,g)$, $\Quord(A,f)\subseteq\Quord(A,g)$.
\end{proposition}

\begin{proof}
  The proof directly follows from Proposition~\ref{B1}\eqref{B1i} and
  Remark~\ref{B1a} applied to
  the residual mapping $\Phi$, cf.~\ref{B2} (the role of $m,q,r$
  in~\ref{B1}\eqref{B1i} and~\ref{B1a} here is played by
  $\Con(A,f),\Quord(A,g),\Quord(A,f)$). 
\end{proof}

We shall describe first the \meet-irreducibles for
permutations $f$ and -- in the next section -- for acyclic $f$.

For permutations $f$
the \meet-irreducible quasiorder lattices $\Quord(A,f)$ are
known. They are described in \cite[Theorem~3.2]{JakPR2016}: $f$ is
either a transposition or of the form as given in Theorem~\ref{Thm3}
below. We first exclude the transpositions:

\begin{lemma}\label{P4}
  Let $f\in A^{A}$ be a transposition $(|A|\geq 3)$. Then $\Con(A,f)$
  is not \meet-irreducible. 
\end{lemma}
\begin{proof}
  If $f\in A^{A}$ is a transposition, then there are elements $0,1\in
  A$ such that $f0=1$, $f1=0$ and $fx=x$ for $x\in A\setminus
  \{0,1\}$. Let $g_{0}, g_{1}\in A^{A}$ be the nontrivial functions defined by
  $g_{0}0=g_{0}1=0$, $g_{1}0=g_{1}1=1$ and all $x\in A\setminus
  \{0,1\}$ are fixed point for $g_{0}$ and $g_{1}$. 
Then, for the principal congruences, we have
$\theta_{f}(x,y)=\theta_{g_{0}}(x,y)=\theta_{g_{1}}(x,y)=[x,y]$ for
all $x,y\in A$ with the only exceptions
$\theta_{f}(0,x)=\theta_{f}(1,x)=\theta_{g_{0}}(1,x)=\theta_{g_{1}}(0,x)=[0,1,x]$
for all $x\in A\setminus\{0,1\}$. 
Note $\theta_{f}(0,x)\neq \theta_{g_0}(0,x)$ and $\theta_{f}(1,x)\neq \theta_{g_1}(1,x)$.
Therefore $\Con(A,f)\subsetneqq \Con(A,g_{i})$, $i\in\{0,1\}$ and
$\Con(A,f)=\Con(A,g_{0})\cap\Con(A,g_{1})$. 
\end{proof}

The following proposition deals with those functions which will play the
crucial role in the next theorem.

\begin{proposition}\label{P1a}
Let $|A|\geq 3$ and let $f\in\Sym(A)$
    be a permutation of prime power order $p^{m}$ with at least two
    cycles of length $p^{m}$. Then the principal filter
    \begin{align*}
      [\Con(A,f)\rangle_{\cE}:=\{E\in\cE\mid \Con(A,f)\subseteq E\}
    \end{align*}
is a chain. Moreover, each element of this chain is
\meet-irreducible (except the top element $\Eq(A)$) and is of the form $\Con(A,g)$, where $g=f^{k}$ for some
$k\in\N_{+}$. 
  
\end{proposition}

\begin{proof}
  Given $f$ as indicated we know from
  \cite[Theorem~4.2 and 4.3]{JakPR2016} that the principal filter
  $[\Quord(A,f)\rangle_{\cL}$ is a chain. From Lemma~\ref{A3a} and
  Corollary~\ref{B6} 
  we conclude
  $\Phi([\Quord(A,f)\rangle_{\cL})=[\Con(A,f)\rangle_{\cE}$,
  consequently $[\Con(A,f)\rangle_{\cE}$ is also a chain, therefore
  each element (except $\Eq(A)$) of this chain is \meet-irreducible
  and thus of the form 
  $\Con(A,g)$. By~\ref{A3} each 
    nontrivial $g$ with $\Con(A,f)\subseteq\Con(A,g)$ is of the form
    $g=f^{k}$.
\end{proof}



\begin{theorem}\label{Thm3}

A congruence lattice $\Con(A,f)$ with a nontrivial permutation $f$ is
\meet-irreducible in $\cE$ if and only if $f$ is
of prime power order $p^{m}$ with at least two cycles
of length $p^{m}$.
\end{theorem}

\begin{proof} ``$\Leftarrow$'' was proved in ~\ref{P1a}.
 
``$\Rightarrow$'': Let $f$ be a permutation such that $\Con(A,f)$ is
\meet-irreducible. By~\ref{P3} there exists $g\in A^{A}$ such that 
$\Quord(A,g)$ is \meet-irreducible in $\cL$ and
\begin{align*}
  \text{(*) } \Quord(A,f)\subseteq\Quord(A,g), 
        && \text{(**) } \Con(A,f)=\Con(A,g).
\end{align*}

As shown in  \cite[Lemma~2.4(iv)]{JakPR2016}), from (*) follows that
each subalgebra of $(A,f)$ is also a subalgebra of $(A,g)$, while from
(**) and Lemma~\ref{A2}\eqref{A2ii} (interchange the roles of
$f$ and $g$) follows that each subalgebra of $(A,g)$
 with at least two elements is also a subalgebra of $(A,f)$ (since $f$ is a
permutation, it cannot be constant on two elements).

At first we show that $g$ is also a permutation. Let $x,y\in
A$ such that $x\neq y$. Since cycles (what coincides with components)
of $f$ are subalgebras of $(A,f)$ 
and thus also of $(A,g)$,
$K_{f}(x)\neq K_{f}(y)$ implies $gx\in K_{f}(x)$, $gy\in K_{f}(y)$,
hence $gx\neq gy$. 
 Thus let $x,y$ belong to the same cycle of $f$ and assume
$gx=gy$. Then from (**) and Lemma~\ref{A2}\eqref{A2i} (here the roles
of $f$ and $g$ are interchanged) we conclude 
$(fx,fy)\in \theta_{g}(x,y)=[x,y]$, thus $\{fx,fy\}=\{x,y\}$ is a
subalgebra of $(A,f)$ and therefore $\{x,y\}\leq (A,g)$ (as
mentioned above). Thus
w.l.o.g.\ we can assume $gx=gy=x$. Let $z\in A\setminus\{x,y\}$ and
let $C$ be the cycle of $f$ which contains $z$. Then $\{x\}\cup C$ is
a  subalgebra(with at least $2$-elements) of $(A,g)$ but not of
$(A,f)$ (since $fx=y\notin\{x\}\cup C$), a contradiction.

Thus $g$ is a permutation. Therefore,
from~\cite[Proposition~2.5(b)]{JakPR2016} and \meet-irreducibility of
$\Quord(A,g)$ we get that $g$ is a permutation of prime power order
$p^{m}$ with at least two cycles of length $p^{m}$ or that $g$ is a
transposition. Since $\Con(A,g)=\Con(A,f)$ is \meet-irreducible, $g$ cannot be a
transposition by Lemma~\ref{P4}. From (**)
and Proposition~\ref{A3} (interchange the role of $f$ and $g$) we get
$f=g^{k}$ for some $k\in\N_{+}$. From (**) we further
conclude that $p$ cannot divide $k$ (since
$\Con(A,g^{p})\supsetneqq\Con(A,g)$). Therefore $f$ and $g$
generate the same cyclic subgroup, in particular $f$ also has order
$p^{m}$ and at least two cycles of length $p^{m}$, and we are done.
\end{proof}

\section{\meet-irreducible $\boldsymbol{\Con(A,f)}$ in
  $\boldsymbol{\cE}$ with acyclic $\boldsymbol{f}$}\label{sec:4a}

In this section we deal with acyclic algebras $(A,f)$. For an acyclic
$f\in A^{A}$ and $x\in A$ let $t_{f}(x):=\min\{n\in\N\mid
f^{n}x=f^{n+1}x\}$ denote the so-called \New{depth} of $x$ (after $n$
times applying $f$ to $x$ one gets a fixed point) and let
\[\bar t(F):=\max\{t_{f}(x)\mid x\in A\}\] (this is the length of a
longest ``tail'' in the graph of $f$). 

\begin{subpart}\label{D0}
For a nontrivial, acyclic function $f\in A^{A}$ we consider the following
conditions (cf.\ Figure~\ref{fig:3}):
\begin{itemize}
\item[(a)] There exist distinct elements $0,1,2,0',1',2'\in A$ such
  that $f2=1,\, f1=f0=0$ and $f2'=1',\, f1'=f0'=0'$,
\item[(b)] $f$ is connected (i.e., only one component) and there
  exist distinct elements $0,1,2,1',2'\in A$ such that $f2=1,\,f2'=1',\,
  f1'=f1=f0=0$.
\end{itemize}

\begin{figure}[htb]
  \centering
  \scalebox{0.8}{\input{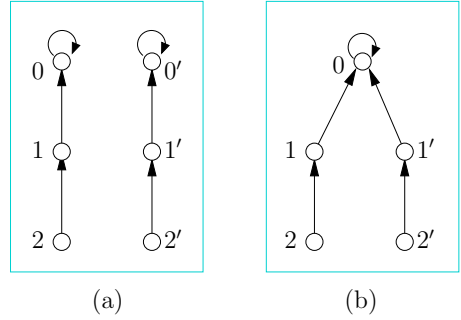}}
  \caption{The action of $f$ on the elements in
    conditions~\ref{D0}(a),(b)}
  \label{fig:3}
\end{figure}

\end{subpart}

\begin{proposition}\label{D1a}
Let $f\in A^{A}$ be nontrivial and acyclic such that $f$ is not of
type {\rm (I)}, not of type {\rm (II)} and satisfies
 neither condition {\rm\ref{D0}(a)} nor {\rm(b)}. Then $\Con(A,f)$ is
  \meet-reducible.
\end{proposition}

\begin{proof}
If  $\bar t(f)=1$ then $f$ is of type (I). 
Thus we can assume $\bar t(f)\geq 2$. We distinguish the following cases:

Case 1: $f$ has at least two components.\\
Then (b) trivially does not hold.
If (a) fails to hold, then $f$ has exactly one component, say $K$, with
elements of depth $2$ while all other components have elements of
depth at most $1$. In particular there are at least two fixed points,
say $0\in K$ and $0'$.
Therefore $f$ is of the form as given in Figure~\ref{fig:5a} (the
shadowed part is $K$).

\begin{figure}[htb]
  \centering
\scalebox{0.69}{\input{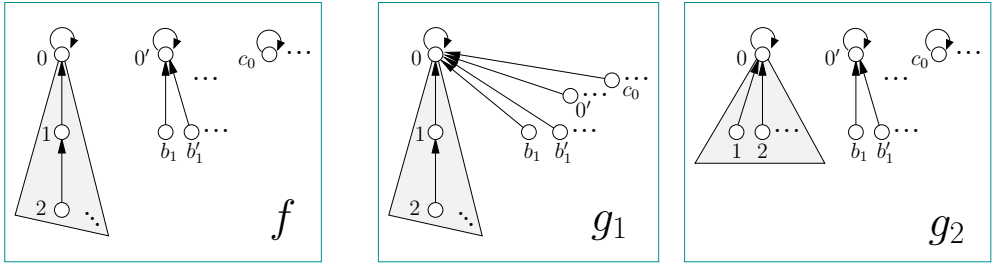}}
  \caption{Functions $f,g_{1},g_{2}$ with
    $\Con(A,f)=\Con(A,g_{1})\cap\Con(A,g_{2})$}
  \label{fig:5a}
\end{figure}

We define the functions $g_{1}$ and $g_{2}$ as follows (see Figure~\ref{fig:5a}):
\begin{align*}
  g_{1}x:=
  \begin{cases}
    fx &\text{ if } x\in K,\\
    0  &\text{ otherwise.}
  \end{cases}
&&
 g_{2}x:=
  \begin{cases}
    0 &\text{ if } x\in K,\\
    fx &\text{ otherwise.}
  \end{cases}
\end{align*}

Case 2: $f$ has only one component with fixed point, say $0$.\\
Then (a) trivially does not hold. 
If (b) fails to hold then all elements $x$ with $t_{f}(x)=2$ map to
the same element, say $1$, of depth $1$.

Case 2a: If $\bar t(f)=2$, then there is only one element of depth
$1$, because otherwise $f$ would be of type (II) what is excluded by
assumption. Therefore $f$ is of the form as given in
Figure~\ref{fig:5b}.

\begin{figure}[htb]
  \centering
\scalebox{0.7}{\input{JPR2015fig5b.pspdftex}}
  \caption{Functions $f,g_{1},g_{2}$ with
    $\Con(A,f)=\Con(A,g_{1})\cap\Con(A,g_{2})$}
  \label{fig:5b}
\end{figure}

We define the functions $g_{1}$ and $g_{2}$ as follows (see Figure~\ref{fig:5b}):
\begin{align*}
  g_{1}x&:=
  \begin{cases}
    0 &\text{ if } x=0,\\
    1 &\text{ otherwise.}
  \end{cases}
&
 g_{2}x&:=
  \begin{cases}    
    1 &\text{ if } x=1,\\
    0 &\text{ otherwise.}
  \end{cases}
\end{align*}

Case 2b: If $\bar t(f)\geq 3$ then $f$ must be of the form as given in
Figure~\ref{fig:5c}.

\begin{figure}[htb]
  \centering
\scalebox{0.7}{\input{JPR2015fig5c.pspdftex}}
  \caption{Functions $f,g_{1},g_{2}$ with
    $\Con(A,f)=\Con(A,g_{1})\cap\Con(A,g_{2})$}
  \label{fig:5c}
\end{figure}

We define the functions $g_{1}$ and $g_{2}$ as follows (see Figure~\ref{fig:5c}):
\begin{align*}
  g_{1}x&:=
  \begin{cases}
     1 &\text{ if } t_{f}(x)\geq 2,\\
     0 &\text{ otherwise.}
  \end{cases}
&
 g_{2}x&:=
  \begin{cases} 
    0 &\text{ if } t_{f}(x)=2,\\
   fx &\text{ otherwise.}
  \end{cases}
\end{align*}

 In all cases the functions $g_{1},g_{2}$ are nontrivial 
and it is easy to check that these functions satisfy
$\Con(A,f)\subsetneqq \Con(A,g_{i})$ ($i\in\{1,2\}$)
and $\Con(A,f)=\Con(A,g_{1})\cap \Con(A,g_{2})$. In
fact,  using the 
Figures~\ref{fig:5a}, \ref{fig:5b} and \ref{fig:5c}, we can check
$(g_{i}x,g_{i}y)\in \theta_{f}(x,y)$, therefore 
$\Con(A,f)\subseteq \Con(A,g_{i})$
 by Remark~\ref{A2a}; moreover the inclusions are
 strict (e.g.\ for Case 1 
we have $[1,b_{1}]=\theta_{g_{1}}(1,b_{1})\neq
\theta_{f}(1,b_{1})=[1,b_{1}][0,0']$). Further, 
$\theta_{f}(x,y)=\theta_{g_{1}}(x,y)\lor \theta_{g_{2}}(x,y)$ for all
$x,y\in A$ (e.g.\ for Case 1 we have
$\theta_{f}(2,b_{1})=[2,b_{1}][1,0,0']=[2,b_{1}][1,0] \lor
[2,b_{1}][0,0']= \theta_{g_{1}}(2,b_{1})\lor \theta_{g_{2}}(2,b_{1})$,
cf.~Figure~\ref{fig:5a}). Consequently, $\Con(A,f)=\Con(A,g_{1})\cap
\Con(A,g_{2})$. Thus $(A,f)$ is \meet-reducible. 
\end{proof}

\begin{proposition}\label{D1b}
  Let $f\in A^{A}$ be acyclic such that $f$ is not of type {\rm(I)},
  not of type {\rm
  (II)} and does satisfy either condition {\rm\ref{D0}(a)} or {\rm(b)}. 
  Let $\Con(A,f)\subsetneqq\Con(A,g)$ for $g\in A^{A}$. Then we have 
  $\rho_{1}\in \Con(A,g)$ for the equivalence relation $\rho_{1}:=[0,2]$.
\end{proposition}

\begin{proof} Assume $\rho_{1}\notin \Con(A,g)$ and we shall show that
  this leads to a contradiction. If $f$ satisfies (b), it is
  convenient to put formally $0':=0$.

We have $\bar t(f)\geq 2$, because $\bar t(f)=1$ means that $f$ is of
type (I).

\underline{Claim 1}: \emph{$\{0,1,2\}\leq (A,g)$}.\\
In fact, if $\{0,1,2\}$ were not a subalgebra, then
(by~\ref{A2}(\ref{A2ii})) $g$ would be constant on $\{0,1,2\}$. Thus
$\rho_{1}=[0,2]=\theta_{g}(0,2)\in \Con(A,g)$, a
contradiction.

\underline{Claim 2}: \emph{$\{0,1\}\leq (A,g)$}.\\
Assume that $\{0,1\}$ is not a subalgebra. Then  $g$
must be constant on $\{0,1\}$ (by~\ref{A2}(\ref{A2ii})) where the
constant is outside $\{0,1\}$. Because of Claim~1
we get $g0=g1=2$ and $g2\in\{0,1,2\}$. The values
$g2\in\{0,2\}$ cannot appear (otherwise $\rho_{1}=[0,2]\in\Con(A,g)$),
thus it remains $g2=1$. Since $\{0,0',1,1'\}\leq (A,f)$, again
by~\ref{A2}(\ref{A2ii}) $g$ must be constant $2$ on these elements, in
particular $g1'=2$. Consequently 
$(1,2)=(g2,g1')\in\theta_{f}(2,1')=[2,1'][1,0',0]$, a contradiction
 (see Figure~\ref{fig:3}).

\underline{Claim 3}: We have \emph{$\{0\}\leq (A,g)$}, i.e. $g0=0$.\\
If $g0\neq 0$, then $g0=1$ (by Claim~2).  
Because $1\notin\{0,0',1',2'\}\leq (A,f)$ (recall $0'=0$ in case (b)) and 
by~\ref{A2}(\ref{A2ii}) $g$ must be constant $1$ on these elements, in
particular $g2'=1$. 
Thus $(g2,1)=(g2,g2')\in\theta_{f}(2,2')=[2,2'][1,1'][0,0']$ implies
$g2=1$ according to Claim 1. Consequently $\theta_{g}(0,2)=[0,2]=\rho_{1}$, a
contradiction. Thus $g0=0$.

\underline{Claim 4}: $g$ and $f$ agree on $\{0,1,2\}$.\\
Because $g0=0$ (by claim~3), the values $g2\in\{0,2\}$ cannot appear 
(otherwise $\rho_{1}=[0,2]\in\Con(A,g)$). Thus $g2=1$ (by Claim~1).
It remains to prove $g1=0$. If $g1\neq 0$ then $g1=1$ (by
Claim~2). Thus $(1,g2')=(g1,g2')\in\theta_{f}(1,2')=[1,2'][0,1',0']$
what implies either $g2'=2'$  or $g2'=1$, 
the former gives the contradiction
$(1,2')=(g2,g2')\in\theta_{f}(2,2')=[2,2'][1,1'][0,0']$ and the latter
gives the contradiction $(0,1)=(g0,g2')\in\theta_{f}(0,2')=[0,0',1',2']$. Thus
$g1=0$. 

\underline{Claim 5}:   $g$ and $f$ agree on
$\{0',1',2'\}$.\\ 
From $(1,g2')=(g2,g2')\in\theta_{f}(2,2')=[2,2'][1,1'][0,0']$ we
conclude $g2'\in\{1,1'\}$. However $g2'=1$ gives the contradiction 
as seen in Claim~4. Consequently $g2'=1'$. Thus $\{0',1',2'\}$ is a
subalgebra of $(A,g)$ (since $g2'\in\{0',1',2'\}\leq(A,f)$). Further, 
$(0,g1')=(g1,g1')\in\theta_{f}(1,1')=[1,1'][0,0']$ gives
$g1'\in\{0',0\}$, thus $g1'=0'=f1'$ (note $g1'$ must belong to the
subalgebra $\{0',1',2'\}$). Finally,
$(0,g0')=(g0,g0')\in\theta_{f}(0,0')=[0,0']$ implies $g0'=0'=f0'$
(note $g0'$ must belong to the subalgebra $\{0',1'\}$).

\underline{Claim 6}: We have
$Z_{f}(x)\leq(A,g)$, i.e. $Z_{g}(x)\subseteq Z_{f}(x)$, for 
each $x\in A$.\\ 
If $f$ satisfies \ref{D0}(b), then $0\in Z_{f}(x)\leq (A,f)$, and with
\ref{A2}(\ref{A2ii}) and $g0=0$ we get $Z_{f}(x)\leq(A,g)$.\\
If $f$ satisfies \ref{D0}(a), then
 $B:=\{0\}\cup Z_{f}(x)$ and $B':=\{0'\}\cup Z_{f}(x)$ are subalgebras
of $(A,f)$ and $g0=0$, $g0'=0'$. Thus  they are also subalgebras of $(A,g)$
(due to \ref{A2}(\ref{A2ii})),
consequently the intersection $B\cap B'=Z_{f}(x)$ is also a subalgebra
of $(A,g)$.

\underline{Claim 7}:   $g=f$.\\
Because of Claim 4 and 5 we have to show $gx=fx$ for each $x\in
A\setminus \{0,1,2,0',1',2'\}$.  
If $x$ is a fixed point of $f$, then $Z_{f}(x)=\{x\}$ and from
Claim~6 we get $gx=x=fx$. Thus let $t_{f}(x)\geq 1$.
W.l.o.g. we can assume $1\notin Z_{f}(x)$ 
(otherwise interchange the role of $1$ and $1'$).  We have
\begin{align*}
(1,gx)=(g2,gx)\in\theta_{f}(2,x)=
\begin{cases}
  [2,x][1,fx][0,f^{2}x,\dots,f^{k}x] &\text{ if } t_{f}(x)\geq 2,\\
  [2,x][0,1,fx] &\text{ if } t_{f}(x)=1.
\end{cases}  
\end{align*}
If $t_{f}(x)\geq 2$ we conclude $gx\in\{1,fx\}$ and get
$gx=fx$ (since $1\notin Z_{g}(x)$).
If $t_{f}(x)=1$, then we conclude $gx\in\{0,1,fx\}$. Moreover, by
assumption we have $x\notin\{0,1,fx\}$ and, by Claim~6 we get $gx\in
Z_{f}(x)=\{x,fx\}$. Thus $gx\in \{x,fx\}\cap\{0,1,fx\}=\{fx\}$. 

From Claim~7 
we get $\Con(A,f)=\Con(A,g)$, a contradiction.
\end{proof}

\begin{theorem}\label{Thm4}
  A congruence lattice $L=\Con(A,f)$ with an acyclic $f\in A^{A}$ is
\meet-irreducible in $\cE$ if and only if $f$ is of type {\rm (I)} or
{\rm (II)} or satisfies the condition {\rm \ref{D0}(a)} or {\rm (b)}.
\end{theorem}

\begin{proof} Note that $\Con(A,f)$ is a coatom for functions $f$ of type
  (I) or (II), and therefore \meet-irreducible. So we need not
  consider these cases in the following.

``$\Rightarrow$'': follows from Proposition~\ref{D1a}

``$\Leftarrow$'':  From Proposition~\ref{D1b} we
conclude that 
  \begin{align*}
    \rho_{1}\in\bigcap\{\Con(A,g)\mid \Con(A,f)\subsetneqq\Con(A,g)\}.
  \end{align*}
Since
$\theta_{f}(0,2)=[0,1,2]$ we have  $\rho_{1}\notin\Con(A,f)$ and 
the above intersection cannot be equal to $\Con(A,f)$. Therefore
$\Con(A,f)$ is \meet-irreducible.
\end{proof}

\section{Some lattice theoretical properties of $\cE$}
\label{sec:5}

At first we consider the problem how many coatoms (atoms, resp.) do we
need such that their meet (join, resp.) in $\cE$ gives the least
(greatest, resp.) element of $\cE$. We assume throughout that $|A|\geq 3$.

\begin{proposition}\label{P7}
There are two or three coatoms in the lattice $\cE$ whose meet is~$\LE$.
More precisely, for $|A|>4$, there are two coatoms
$\Con(A,f)$ and $\Con(A,g)$ such that
 $\Con(A,f)\cap \Con(A,g) =\{\Delta,\nabla\}$.
For $|A|\leq 4$, three coatoms are necessary \textup(and sufficient\textup) for this
property.
\end{proposition}

\begin{proof}
  Since $\Con(A,f)=\Eq(A)\cap\Quord(A,f)$ is a coatom in $\cE$ iff $\Quord(A,f)$ is a coatom
  in $\cL$ (cf.~\ref{B4}, \ref{P2}), the result mainly follows
  from the corresponding result
  in~\cite[Proposition~6.2]{JakPR2016}. In that paper for $|A|>5$
  there are indicated two
  permutations $f$ and $g$ of type (III), for $|A|=5$ three
  permutations of type (III), for $|A|=4$ one permutation of type
  (III) and two functions of type (I) and for $|A|=3$ three functions
  of type (I). Now, with $\cE$ instead of $\cL$, in case
  $A=\{1,\dots,n\}$ for $n=5$ also two functions (e.g., of 
  type (II)) suffice, e.g., $f4=2$, $f5=3$, $f2=f3=f1=1$ and
  $g1=4$, $g2=5$, $g4=g5=g3=3$.
  It can be checked easily that two coatoms are not sufficient for
  $n\in\{3,4\}$.  
\end{proof}

\begin{proposition}\label{P8}
There are three atoms in $\cE$ whose join is
$\GE$. More precisely, there are three equivalence relations
$\kappa_{1}, \kappa_{2}, \kappa_{3}$ such that
$E_{\kappa_{1}}\vee E_{\kappa_{2}}\vee E_{\kappa_{3}}=\Eq(A)$.
\end{proposition}

\begin{proof}
  By a result of L.~Z\'adori~\cite{Zad86} there exist equivalence
  relations $\kappa_{1}, \kappa_{2}, \kappa_{3}$ such that
  $\End\{\kappa_{1},\kappa_{2},\kappa_{3}\}=\{\id_{A}\}$. This implies
$E_{\kappa_{1}}\vee E_{\kappa_{2}}\vee
E_{\kappa_{3}}=\Con\End(E_{\kappa_{1}}\cup E_{\kappa_{2}}\cup
E_{\kappa_{3}})=\Con\{\id_{A}\}=\Eq(A)$. (For notation $E_{\kappa}$
see \ref{Thm2}.)
\end{proof}

Now we look for tolerances of the lattice $\cE$. Because tolerance simplicity
implies interesting properties of a lattice (see, e.g.,
\cite{Kin79}), we looked for this property for the lattice $\cE$ and got an
affirmative result in Theorem~\ref{P6b} below.

At first we collect some notions, notations and facts which for clearer understanding we shall present on abstract level
(for an arbitrary lattice $V$ instead of our lattice $\cE$).

\begin{subpart}\label{P6bA}
Let $V$ be a lattice with the order and covering relation denoted by
$\leq$ and $\cover$, respectively. If $V$ is a bounded lattice (in
particular, if it is finite), its least and greatest elements are
denoted by $\LE[V]$ and $\GE[V]$.

A \New{tolerance} of  $V$ is a reflexive and symmetric
binary relation $T\subseteq V\times V$ compatible with the lattice
operations \meet\ and \join. Let $\Tol(V)$ denote all tolerances of $V$.
 With respect to set-theoretic inclusion the
tolerances form an algebraic lattice $(\Tol(V),\cap,\sqcup)$ with
least element $\LT[V]:=\{(x,x)\mid x\in V\}$ and greatest
element $\GT[V]:=V\times V$ (called \New{trivial} tolerances).
A lattice $V$ is called \New{tolerance simple} if it has no nontrivial
tolerances, i.e., $\Tol(V)=\{\LT[V],\GT[V]\}$.

For $x,y\in V$, let $T(x,y)$ denote the least tolerance in
$\Tol(V)$ containing the pair $(x,y)$. Clearly, for each $T\in\Tol(V)$,
we have $T=\bigsqcup\{T(x,y)\mid (x,y)\in T\}$.  The following properties are
known (see, e.g., \cite{RadS05}) for $x,y\in V$:
\begin{align}
  T(x\meet y,y)=T(x,x\join y), 
   \label{t1}\\
  (\LE[V],\GE[V])\in T\in\Tol(V)\implies T=\GT[V]. \label{t2}
\end{align}

A lattice $V$ is called \New{atomistic} if every element $v\in
V\setminus\{\LE[V]\}$ is the join of some atoms of $V$. The atoms of
$V$, denoted by $\At(V)$ in the following, play an important role also
in connection with tolerance simplicity. From \cite{JanR2015}
we deduce (see also \cite[6.4]{JakPR2016}) the following: A finite
atomistic lattice $V$ satisfying $T(\LE[V],a)=\GT[V]$ for every atom $a\in
V$, is tolerance simple.
\end{subpart}

\begin{lemma}\label{P6bB} Let $V$ be a finite atomistic lattice. Then
  we have:
  \begin{enumerate}[\rm(i)]
  \item \label{P6bBi} 
    Let $a_{1},a_{2}\in\At(V)$, $a_{1}\neq a_{2}$ and let $d\in V$
    be a coatom such that $a_{1}\not\leq d$, $a_{1}\not\leq d$. Then 
    $T(\LE[V],a_{1})=T(\LE[V],a_{2})$.
  \item \label{P6bBii} 
    If $T(\LE[V],a_{1})=T(\LE[V],a_{2})$ for all
    $a_{1},a_{2}\in\At(V)$, then $V$ is tolerance-simple.
  \end{enumerate}
  
\end{lemma}
\begin{proof}
\eqref{P6bBi}: Since $a_{i}\meet d=0$ and $a_{i}\join d=1$ for
$i=1,2$, we get
\begin{align*}
  T(\LE[V],a_{i})=T(d\meet a_{i},a_{i})=_{\eqref{t1}}T(d,d\join
  a_{i})=T(d,\GE[V]),
\end{align*}
consequently, $T(\LE[V],a_{1})=T(\LE[V],a_{2})$.

\eqref{P6bBii}: Since $T(\LE[V],a)$ is the same tolerance for each atom
$a\in\At(V)$, we will denote it by $\alpha$. We have
$(\LE[V],a)\in\alpha$ for all $a\in\At(V)$, consequently
$(\LE[V],\GE[V])=(\LE[V],\bigvee\At(V))\in\alpha$. From \eqref{t2} we get
$\alpha=\nabla_{V}$, i.e., $T(\LE[V],a)=\nabla_{V}$
for all $a\in\At(V)$. As mentioned above in~\ref{P6bA}, this implies
tolerance-simplicity of $V$.
\end{proof}

Now, instead of the abstract lattice $V$, we return to the concrete lattice
$\cE$. Recall that $\cE$ is atomistic and
$\At(\cE)=\{E_{\kappa}\mid \kappa\in\Eq(A)\}$ where
$E_{\kappa}=\{\Delta,\kappa,\nabla\}$ (Theorem~\ref{Thm2}). The least
and greatest elements are $\LE=\{\Delta,\nabla\}$ and $\GE=\Eq(A)$.
As defined in~\ref{A1a}, $[a,b]$ denotes the equivalence relation (on $A$) with
one nontrivial block $\{a,b\}$.

\begin{lemma}\label{P6bC}
Let $\kappa\in\Eq(A)\setminus\{\Delta,\nabla\}$. Then there exists
$(a,b)\in\kappa$, $a\neq b$, such that $T(\LE,E_{\kappa})=T(\LE,E_{[a,b]})$.
  
\end{lemma}

\begin{proof} Since $\kappa$ is nontrivial there exist distinct
  elements $a,b,c\in A$ such that $(a,b)\in\kappa$ but
  $(a,c)\notin\kappa$.  Clearly $(a,b)\in[a,b]\subseteq
  \kappa$. Define $f\in A^{A}$ via $fx:=c$ if $x=a$, and $fx=x$
  otherwise. Then $f=f^{2}$ is of type (I) and therefore $\Con(A,f)$
  is a coatom in $\cE$ (cf.~Theorem~\ref{Thm1}). Since
  $(fa,fb)=(c,b)$, the function $f$
  preserves neither $\kappa$ nor $[a,b]$. Hence
  $\kappa,[a,b]\notin\Con(A,f)$, consequently the atoms
  $E_{\kappa},E_{[a,b]}$ are not contained in $\Con(A,f)$ and
  therefore, by applying Lemma~\ref{P6bB}\eqref{P6bBi}, we
  obtain $T(\LE,E_{\kappa})=T(\LE,E_{[a,b]})$.
\end{proof}

\begin{theorem}\label{P6b}
For $|A|\geq 4$, the lattice $\cE$ is tolerance simple.
\end{theorem}

\begin{proof} In view of Lemma~\ref{P6bB}\eqref{P6bBii} it is
  sufficient to show \[T(\LE,E_{\kappa_{1}})=T(\LE,E_{\kappa_{2}})
  \text{ for all atoms } E_{\kappa_{1}},E_{\kappa_{2}},\] i.e., for all
  $\kappa_{1},\kappa_{2}\in\Eq(A)\setminus\{\Delta,\nabla\}$. Due to
  Lemma~\ref{P6bC} we even may restrict to equivalence relations of
  the form $\kappa_{1}=[a_{1},b_{1}]$, $\kappa_{2}=[a_{2},b_{2}]$ for
  $(a_{1},b_{1}), (a_{2},b_{2})\in A^{2}\setminus\Delta$. If
  $\{a_{1},b_{1}\}=\{a_{2},b_{2}\}$, then
  $E_{[a_{1},b_{1}]}=E_{[a_{2},b_{2}]}$ and we are done. Hence,
  w.l.o.g, we can restrict to the following two cases:

\textsl{Case} (a): $a_{1}=a_{2}$ and $b_{1}\neq b_{2}$.\\
Since $|A|\geq 4$ there exists an element $c\in
A\setminus\{a_{1},b_{1},b_{2}\}$. Define $f\in A^{A}$ by $fx=c$ if
$x\in\{b_{1},b_{2}\}$ and $fx=x$ otherwise. Then $f$ is of type (I)
and $\Con(A,f)$ is a coatom (cf.\ Theorem~\ref{Thm1}). From the
definition immediately follows that $f$
does not preserve neither $\kappa_{1}=[a_{1},b_{1}]$ nor
$\kappa_{2}=[a_{2},b_{2}]$. Consequently $E_{\kappa_{1}}\not\subseteq
\Con(A,f)$, $E_{\kappa_{2}}\not\subseteq\Con(A,f)$ and from
Lemma~\ref{P6bB}\eqref{P6bBi} we conclude
$T(\LE,E_{\kappa_{1}})=T(\LE,E_{\kappa_{2}})$. 

 \textsl{Case} (b): $\{a_{1},b_{1}\}\cap\{a_{2},b_{2}\}=\emptyset$.\\
Consider the permutation $f:=(a_{1}a_{2})(b_{1}b_{2})$ (with two
cycles of length $2$). By Theorem~\ref{Thm1}, $\Con(A,f)$ is a coatom
(type (III)) and we have $f\not\preserves\kappa_{1}$,
$f\not\preserves\kappa_{2}$ and as in case (a) above we get $T(\LE,E_{\kappa_{1}})=T(\LE,E_{\kappa_{2}})$.
\end{proof}

The investigation of lattice properties around modularity shows that
such properties cannot be expected for $\cE$:

\begin{proposition}\label{P6a}
For $|A|\geq 4$, the lattice $\cE$ has none of the following properties:
$0$-modular, $1$-modular, lower semimodular, upper semimodular.
\end{proposition}

\begin{proof}
If $A$ has at least $4$ elements, say $0,1,2,3$, then consider the
nontrivial equivalence relations $\kappa_{1}=[0,1,2]$,
$\kappa_{2}=[0,1][2,3]$, $\kappa_{0}=\kappa_{1}\cap\kappa_{2}=[0,1]$
 and the function $f\in A^{A}$ defined by $fx=3$ if $x=0$, and $fx=x$ otherwise.

\begin{figure}[htb]
\begin{center}
  \input{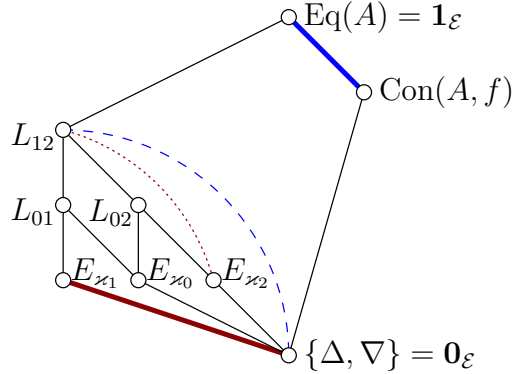}
\end{center}
  \caption{The sublattice of $\cE$ used in the proof of Proposition~\ref{P6a}}\label{fig:7}
\end{figure}

We get the sublattice as shown in Figure~\ref{fig:7}, e.g.,
$\Con(A,f)$ is a coatom by Theorem~\ref{Thm1} and it is easy to check
that for $L_{ij}:=E_{\kappa_{i}}\vee E_{\kappa_{j}}$ ($0\leq i<j\leq
2$) we have
$L_{12}=\{\Delta,\kappa_{0},\kappa_{1},\kappa_{2},\nabla\}$ and
 $L_{0j}=\{\Delta,\kappa_{0},\kappa_{j},\nabla\}$ ($j=1,2$).
 Note that $f$
does not preserve neither $\kappa_{0}$ nor $\kappa_{1}$ nor
$\kappa_{2}$.

Obviously $\{\LE,E_{\kappa_{1}},L_{12},\Con(A,f),\GE\}$ is a
sublattice isomorphic to $N_{5}$. Thus $\cE$ is neither $0$- nor
$1$-modular. Further, $\Con(A,f)\cover\GE$ (blue line) but
the meet with $L_{12}$ (dashed blue line) is not
a covering; likewise $\LE\cover E_{\kappa_{1}}$ (red line) but the join with
$E_{\kappa_{2}}$ (dotted red line) is not a covering. Hence $\cE$ is
neither lower nor upper semimodular.
\end{proof}

\begin{remark}
  For $|A|=3$, Theorem~\ref{P6b} and Proposition~\ref{P6a} do not
  remain valid. In this case, $\cE$ is the lattice of all subsets of
  $\{\Delta,\nabla,\theta_{0}, \theta_{1},\theta_{2}\}$ containing $\Delta,\nabla$, where
  $\theta_{i}$, $i\in\{0,1,2\}$, are the nontrivial equivalence
  relation on $A$. Thus it is a Boolean lattice with $8$
  elements. Therefore it is modular and it is not tolerance-simple (it
  is even not congruence-simple). 
\end{remark}

\subsection*{Acknowledgment} The authors thank S.\,Reichard for some
hints and computations which led to an improved version of
Pro\-po\-sition~\ref{P7}. 

We dedicate our paper to the memory of E. Tam\'as
Schmidt who passed away in 2016. The final
version will appear in a special issue of \textsc{Algebra Universalis}.
 Each of the authors is grateful to have known Tam\'as and to
profit from numerous inspiring and friendly discussions with this
outstanding mathematician. 
In particular, the third author is indepted
to Tam\'as as PhD supervisor introducing him to
contemporary lattice theoretical investigations and giving
constructive hints concerning the finite representation problem.

\normalsize

\small\footnotesize

\end{document}